\documentclass[12pt]{article}

\usepackage{mathtools}
\usepackage[demo]{graphicx} 
\usepackage{hyperref}
\usepackage{multirow}
\usepackage{float}
\usepackage{diagbox}
\usepackage{indentfirst}
\usepackage{tikz}
\usepackage{dsfont}
\usepackage{amssymb}
\usepackage{amsmath,mathrsfs,amsfonts}
\usepackage{graphics}
\usepackage{amsthm}
\usepackage{setspace}
\usepackage{setspace}
\usepackage[mathscr]{eucal}
\usepackage{caption}
\usepackage{subcaption}
\usepackage{rotating}
\usetikzlibrary{decorations.pathreplacing}
\newtheorem{theo}{Theorem}

\newtheorem{lem}[theo]{Lemma}

\newtheorem{problem}[theo]{Problem}
 \theoremstyle{definition}
\newtheorem{defi}{Definition}
\newtheorem{claim}{Claim}

\theoremstyle{remark}

\newcounter{casenum}[theo]

\newcounter{subcasenum}[theo]

\newcounter{claimnum}[theo]

\allowdisplaybreaks

\begin{document}
\thispagestyle{plain}

\begin{center}
{\Large Connected graphs with a large dissociation number  \\ attaining the minimum spectral radius}
\end{center}
\pagestyle{plain}
\begin{center}
	{
		{\small  Zejun Huang, Chenxi Yang \footnote{  Corresponding author. \\ Email: zejunhuang@szu.edu.cn (Huang), yangchenxi2022@email.szu.edu.cn (Yang)}}\\[3mm]
		{\small   School of Mathematical Sciences, Shenzhen University, Shenzhen 518060, China }\\
		
	}
\end{center}
\begin{center}
\begin{minipage}{140mm}
\begin{center}
{\bf Abstract}
\end{center}

A dissociation set in a graph is a subset of vertices that induces a subgraph of maximum degree at most one, which is a natural  generalization of  the notion of an  independent set.
The dissociation number of a graph is defined as the maximum cardinality of a dissociation set.
This paper studies  the minimum spectral radius of connected graphs with a given order $n$ and a given dissociation number $\psi$. For  $\psi=n-k$ with $k\ge 4$ fixed and $n$ sufficiently large,
we establish both upper and lower bounds for this minimum spectral radius and prove the extremal graphs must belong to a specific graph class.

\vspace{12pt}
{\small
{\bf Keywords:} spectral radius; dissociation number;  independent number
}

\end{minipage}
\end{center}

\section{Introduction}

Let $G = (V(G), E(G))$ be  a simple graph. The degree of a vertex $v \in V(G)$ is denoted by $d_G(v)$ or $d(v)$. The maximum degree of a graph is denoted by $\Delta(G)$. The \emph{distance} between two vertices $u$ and $v$ in $G$ is denoted by $d_G(u, v)$ or   $d(u,v)$. For a subset $S \subseteq V(G)$, $G[S]$ is the subgraph of $G$ induced by $S$, and $G[V(G) \setminus S]$  (or  $G-S$) is the graph obtained from $G$ by deleting the vertices in $S$, which is also written as  $G-v$ when $S = \{v\}$ is a singleton. The \emph{spectral radius} of $G$, denoted by $\rho(G)$, is the spectral radius of its adjacency matrix $A(G)$.

An $n \times n$ matrix $A$ is called \emph{irreducible} if it cannot be permuted to a block upper triangular form via simultaneous row and column permutations. For a nonnegative matrix $A$, it is irreducible if and only if all entries of its associated matrix $(I + A)^{n-1}$ are positive. The notation $A \ge 0$ (resp. $A > 0$) signifies that all entries of the matrix $A$ are nonnegative (resp. positive). Furthermore, the expression $A \ge B$ (resp. $A > B$) is defined by $A-B\ge 0$ (resp. $A-B>0$).

In 1986, Brualdi and Solheid \cite{BS} initiated the study of extremal eigenvalues for matrices within a specific subclass of 0-1 matrices. As the adjacency matrix of a graph is a 0-1 matrix, the corresponding problem for graphs, known as the spectral Brualdi-Solheid problem, has been extensively studied. A significant focus of this research has been on determining the maximum spectral radius under various structural constraints, including independence number, clique number, matching number, diameter, and domination number; see, e.g., \cite{Feng,FENG2007133,HS,JL,SAH,vD}.

The minimization part of the Brualdi-Solheid type problem is considerably more challenging and intricate than its maximization counterpart. For connected graphs with a given independence number $\alpha$, the minimal spectral radius problem has been  solved for specific values of $\alpha$, say, $\alpha \in \{1, 2, \lceil n/2\rceil-1, \lceil n/2\rceil, \lceil n/2\rceil+1, n-6, n-5, n-4, n-3, n-2, n-1\}$ \cite{choi2023minimal, hu2022graphs, LOU2022112778, XU2009937}. Furthermore, Du and Shi \cite{Du2013GraphsWS} resolved the problem for $\alpha = 3$ and $4$ under the additional condition that the order $n$ is divisible by $\alpha$. Pirzada and Bhat \cite{PB} investigated the minimum spectral radius of connected graphs with order $n$ and a matching number $\gamma$ and they solved the cases $\gamma\in \{1,2,3,\lfloor n/2\rfloor-1,\lfloor n/2\rfloor\}$. Moreover, Cui, Chen,  Liu and Wang\cite{CCZ} studied the minimum spectral radius problem for given domination number and independence number in unicyclic graphs.

 A dissociation set in a graph $G$ is a vertex subset $S \subseteq V(G)$ that induces a subgraph of maximum degree at most $1$, which is a natural generalization of the notion of an independent set. The maximum size of a dissociation set is the {\it dissociation number} of $G$, denoted $\operatorname{diss}(G)$. We are interested in the minimal spectral radius of connected graphs of order $n$ with a given dissociation number $\psi$. Denote by $ \mathcal{G}_{n,\psi}$ the connected graphs of order $n$ with dissociation number $\psi$ that attains the minimum spectral radius. We formalize our problem as the following.

\begin{problem}\label{pro1}
Let $n,\psi$ be integers such that $2\le \psi \le n$. Characterize  the structures of graphs in  $ \mathcal{G}_{n,\psi}$.
\end{problem}

Some cases on this problem has been solved.
Huang, Li and Zhan \cite {HLi} solved this problem for the cases $\psi\in \{ 2,\left\lceil 2n/3\right\rceil,n-2,n-1\}$; Huang, Liu and Zhang \cite{HLZ} independently settled the same cases by using a different approach; Zhao, Liu and Xiong \cite{ZLX} solved the problem for the case $\psi=\left\lceil 2n/3\right\rceil-1$; Huang, Liu and Yang \cite{HLY} solved the case $\psi=n-3$.
In this paper, we study the case $\psi=n-k$ for any fixed $k\ge 4$ and sufficiently large $n$.
We establish both upper and lower bounds for the   spectral radius of graphs in $ \mathcal{G}_{n,n-k}$  and prove that $ \mathcal{G}_{n,n-k}$  must belong to a specific graph class.

Let $uv$ be an edge of a graph $G$. The {\it $k$-subdivision} of the edge $uv$ is replacing $uv$ with a  path of length $k$, i.e., deleting $uv$, adding  new vertices $v_1,\ldots,v_k$ and   new edges $uv_1,v_1v_2,\ldots,v_{k-1}v_k,v_kv$. A  1-subdivision  is also called a {\it subdivision}.

Let $G$ be a tree of order $k$.
We denoted by $D_n(G)$ the set of all $n$-vertex graphs obtained from $G$ by:\\
\indent \textbf{(1)} doing a 2-subdivision for each edge;\\
\indent \textbf{(2)} optionally attaching at most one isolated vertex to $V(G)$ and attaching possible additional edges to $V(G)$ such that the degrees of the
original vertices in $V(G)$ differ by at most 1.

Let $\mathcal{T}_k$ be the set of all $n$-vertex trees and denote by $D_n(\mathcal{T}_k)=\cup_{G\in \mathcal{T}_k}D_n(G)$.
Our main result is as follows.

\begin{theo}\label{theorem2}
Let $n$ and $k$ be positive integers such that  $k\ge 4$ and $n\ge 16k^3$, and let   $G\in \mathcal{G}_{n,n-k}$. Then $G\in D_n(\mathcal{T}_k)$ and its  maximum degree is $\lceil (n+k-2)/(2k)\rceil$.
Moreover,  we have $$\rho(H)-\frac{2(k^{3/2}+k)}{3n}\le\rho(G)\le \rho(H) $$  for all   $H\in D_n(\mathcal{T}_k) $ with $\Delta(H)=\lceil (n+k-2)/(2k)\rceil$.
\end{theo}

\section{Preliminaries}
In this section, we present some definitions and lemmas used in the proof.

Given a vertex $v$ and a subgraph $H$ in a graph $G$ such that $v\not\in V(H)$, if $G$ has an edge $uv$ with $u\in V(H)$, we say $H$ is adjacent to $v$.
\begin{defi}\label{defi1}
The {\it skeleton} of a graph $G$ with respect to the vertices $v_1,v_2,\ldots,v_k$, denoted  $S[G;v_1,v_2,\ldots,v_k]$, is obtained from $G$ by deleting all the components of $G-\{v_1,v_2,\ldots,v_k\}$ that are adjacent to exactly one vertex of $v_1,v_2,\ldots,v_k$.
\end{defi}

Suppose $D$ is a maximum dissociation set of $G$. Notice that $G[D]$ consists of isolated vertices and disjoint edges. We always denote the isolated vertices and the disjoint edges in $G[D]$ by $u_1,\dots, u_{\gamma}$ and   $u^1_1u^2_1,\ldots, u^1_{\tau}u^2_{\tau}$, respectively. Since $G$ is connected, every isolated vertex or edge of $G[D]$ is adjacent to at least one vertex of $V(G)\backslash D$.
The following hypergraph is crucial in later structure analysis.

\begin{defi}\label{defi2}
The {\it generated hypergraph} of $G=(V,E)$ with respect to a maximum dissociation set $D$, denoted $H[G,D]$, is the hypergraph  $(V',E')$, where $V'=V(G)\setminus D$ and $E'$ consists of the following three types of hyperedges:\\
\indent \textbf{Type (I)}  the edges in $E$ with both ends in $V'$;\\
\indent \textbf{Type (II)}   $N(u_i)\cap V' $ with $|N(u_i)\cap V'|\ge 2$, $i=1,2,\dots,\gamma$;\\
\indent \textbf{Type (III)}   $N(\{u^1_j,u^2_j\})\cap V'$ with $|N(\{u^1_j,u^2_j\})\cap V'|\ge 2$, $j=1,2,\dots,\tau$.
\end{defi}

A hypergraph $H$ is called \emph{linear\ hypergraph} if every pair of distinct hyperedges in $H$ shares at most one common vertex.

\begin{defi}\label{def4}
Suppose  $G$ is a tree with a dissociation set $D$ such that\\
\indent (1) $H[G,D]$ is a linear hypergraph;\\
\indent (2) all edges of $H[G,D]$ are  {Type (III)}. \\
Denote by $V\setminus D=\{v_1,v_2,\dots, v_k\}$ and let $a_i$ be  the number of leaves attached on $v_i$ for $i=1,2,\ldots,k$.
  For $t>1$, the matrix $M_{G,D}(t)=(m_{ij})_{k\times k}$ induced by $G$ and $D$ is defined as
\[m_{ij}=\left\{
\begin{aligned}
& td_G(v_i)-\frac{a_i}{t}, && \text{if }i=j;\\
& t, && \text{if }d_G(v_i, v_j)=2;\\
& 1, && \text{if }d_G(v_i, v_j)=3;\\
& 0, && \text{otherwise}.
\end{aligned}
\right.
\]
Notice that if $t$ is sufficiently large, we have $\lambda_1(M_{G,D}(t))=O(t)$.
\end{defi}

\begin{lem}\cite{Zhan}\label{PF} If $A$ is an irreducible nonnegative matrix of order $n$ with $n\geq 2$, then the follow statements hold.\\
   \indent (1) $\rho(A)>0$, and $\rho(A)$ is a simple eigenvalue of $A$.\\
   \indent (2) $A$ has a positive eigenvector corresponding to $\rho(A)$.\\
   \indent (3) All nonnegative eigenvectors of $A$ correspond to the eigenvalue $\rho(A)$.
\end{lem}

By the Perron-Frobenius theorem, the largest eigenvalue of a symmetric matrix $A$, denoted by $\lambda_1(A)$, is equal to $\rho(A)$ and associated with a unique unit eigenvector with all components positive. When $A$ is the adjacency matrix of a graph $G$, this eigenvector is called the \emph{Perron vector} of $G$.

\begin{lem}\cite{Zhan}\label{W_ieq} Let $A,B$ be two $n\times n$ symmetric nonnegative matrices. Then
    $$max\{\rho(A),\rho(B)\}\leq \rho(A+B)\leq \rho(A)+\rho(B).$$
\end{lem}

\begin{lem}\cite{Zhan}\label{MS}
If $A,B$ are irreducible nonnegative matrices such that $A\ge B$, then $\rho(A)\ge \rho(B)$ with equality if and only if $A=B$.
\end{lem}
\begin{lem}\label{lemma3}\cite{Hoffman_1975} If $G_2$ is a proper subgraph of a connected graph $G_1$, then $\rho(G_1)>\rho(G_2)$.	
\end{lem}
\par An {\it internal path} of a graph $G$ is a sequence of vertices $v_1, v_2,\ldots, v_k$ with $k\ge2$ such that:\\
\indent (1) the vertices in the sequence are distinct (except possibly $v_1=v_k$);\\
\indent (2) $v_i$ is adjacent to $v_{i+1}, i=1,2,\ldots,k-1$;\\
\indent (3) the vertex degrees satisfy $d(v_1)\ge3$, $d(v_2)=\cdots=d(v_{k-1})=2$ (unless $k=2$) and \indent $d(v_k)\ge3$.

Denoted by  $G^{uv}$ the graph obtained from $G$ by doing a 1-subdivision of $uv$.
\begin{lem}\label{lemma11}\cite{0}
Suppose that $G\ncong \tilde{W}_{n}$ and $uv$ is an edge on an internal path of $G$.  Then $\rho(G^{uv})<\rho(G)$.
\end{lem}

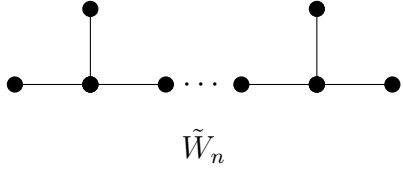
\begin{figure}[H]
		\centering
		\begin{tikzpicture}
			\draw(0,0)[fill]circle[radius=1.0mm]--(1,0)[fill]circle[radius=1.0mm];
			\draw(1,0)[fill]circle[radius=1.0mm]--(2,0)[fill]circle[radius=1.0mm];
			\draw(3,0)[fill]circle[radius=1.0mm]--(4,0)[fill]circle[radius=1.0mm];
			\draw(4,0)[fill]circle[radius=1.0mm]--(5,0)[fill]circle[radius=1.0mm];
			\draw(1,1)[fill]circle[radius=1.0mm]--(1,0)[fill]circle[radius=1.0mm];
			\draw(4,1)[fill]circle[radius=1.0mm]--(4,0)[fill]circle[radius=1.0mm];
			\node at (2.5,0){$\cdots$};
		\end{tikzpicture}
		\caption*{$\tilde{W}_{n}$}
	\label{fig}
\end{figure}

\begin{lem}\label{lemma14}\cite{0} Let $v$ be a vertex in a connected graph  $G$, and let $k, m$ be nonnegative integers such that $k\ge m$. Denote by $G^{k,m}$ the graph obtained from $G$ by attaching   two new paths   $vv_1v_2\cdots v_k$ and  $vu_1u_2\cdots u_m$  to $v$, where $v_1,\ldots, v_k,u_1,\ldots,u_m$ are distinct and $\{v_1,\ldots,v_k,u_1,\ldots,u_m\}\cap V(G)=\emptyset$.  Then  $\rho(G^{k,m})>\rho(G^{k+1,m-1})$.
\end{lem}

\begin{lem}\label{lemma4}\cite{HLi,HLZ}
	Let $n$ and $\psi$ be positive integers with $ \psi>\lceil 2n/3\rceil $. If $G$ attains the minimum spectral radius  in $\mathcal{G}_{n,\psi}$, then $G$ is a tree.
\end{lem}

\begin{lem}\cite{CA,Cvetkovi1995}\label{lemma13}
Let $S_n$ be the star graph of order $n+1$.
Among all connected trees of order $n$,  $S_{n-1}$ has the maximum spectral radius  $\sqrt{n-1}$.
\end{lem}

\begin{lem}\cite{HLY}\label{lemma15}
Let $f(t)$, $h_1(t), h_2(t)$ be  real valued continuous functions for $t\in [a,+\infty)$. Suppose $\rho_i\in [a,+\infty)$ is the maximum root of $f(t)=h_i(t), i=1,2$ and $\rho_0=\min\{\rho_1,\rho_2\}$.
If $h_1(t)>h_2(t)$ for all $t\in [\rho_0,+\infty)$ and $\lim\limits_{t\to\infty}(f(t)-h_i(t))=+\infty$ for $i=1,2$, then $\rho_1>\rho_2$.
\end{lem}

Suppose $A$ is a symmetric real matrix whose rows and columns are indexed by $V=\{1, \ldots, n\}$. Let $\left\{V_1, \ldots, V_m\right\}$ be a partition of $V$.  Denote by  $A_{ij}$ the submatrix of $A$ with row indices from $V_i$ and  column indices  from $V_j$. Let $b_{ij}$ represent the average row sum of $A_{ij}$. The matrix $B=\left(b_{ij}\right)_{m\times m}$ is called the {\it quotient matrix} of $A$. If the row sum of each block $A_{ij}$ is constant, then the partition is called {\it equitable}.

\begin{lem}\label{eqlemma}
\cite{YOU} Let $A$ be a nonnegative matrix and $\pi$ be a partition of $A$ with quotient matrix $B=(b_{i, j})$. Then
$$
\rho(A) \geq \rho(B),
$$
with equality if and only if the partition is equitable. If the partition $\pi$ is equitable, then each eigenvalue of $B$ is an eigenvalue of $A$.
\end{lem}
\begin{lem}\label{claim0}
If $G$ is a connected $n$-vertex graph with ${\rm diss}(G)\le n-1$,  then $\rho(G)> \sqrt{2}$.
\end{lem}
\begin{proof}
Since $G$ is a connected graph and ${\rm diss}(G)<n$,  it contains a $K_{1,2}$. By Lemma \ref{lemma3} and Lemma \ref{lemma13}, we have $\rho(G)>\sqrt{2}$.
\end{proof}

\begin{lem}\label{lem1}
Let $k$ be a positive integer. If an $n$-vertex graph $G$ has $k$ vertices with degree larger than $k$, then ${\rm diss}(G)\le n-k$.
\end{lem}

\begin{proof}
Suppose $D$ is a maximum dissociation set. Since $G$ has a vertex $u$ with $d(u)>k$,  there is a vertex $v_1\in N[u]$ such that $v_1\notin D$.
 Similarly, in the graph $G-v_1$, we can find a vertex $v_2\notin D$, since $G-v_1$ has at least $k-1$ vertices with degree larger than $k-1$. Repeating the above process,
  we can find $v_1,v_2,\dots,v_k\notin D$, which leads to $|D|\le n-k$ and ${\rm diss}(G)\le n-k$.
\end{proof}

\begin{lem}\label{lem2}
If $H\in D_n(P_k)$ with $n=2pk+r$ vertices, $p\ge k, 3k-1\le r\le 5k-2$, then ${\rm diss}(H)=n-k$ and $\rho(H)<1+\sqrt{p+3}$.
Consequently, we have $\rho(G)<1+\sqrt{p+3}$ for $G\in \mathcal{G}_{n,n-k}$.
\end{lem}

\begin{proof}
According to the definition of $D_n(P_k)$, the degree of   each  vertex of $V(P_k)$ in $H$ is at least $2+\lfloor (n-3k-2)/2k\rfloor$. Since
$$2+\lfloor (n-3k-2)/2k\rfloor=2+\lfloor (2pk+r-3k-2)/2k\rfloor\ge 2+\lfloor (2pk-3)/2k\ge p+1> k,$$
$H$ has $k$ vertices with degree larger than $k$. By Lemma \ref{lem1}, we have ${\rm diss}(H)\le n-k$.
 Notice that $V(H)\setminus V(P_k)$ is a dissociation set of $H$ with cardinality $n-k$.   We conclude that ${\rm diss}(H)=n-k$.

 Let
$H'$ be the graph obtained from $C_k$ by attaching $p+1$ additional edges to each  vertex of $V(C_k)$.
Let $H''$ be the graph obtained from $H'$ by doing a 2-subdivision for each edge in the cycle $C_k$, whose order is $(2p+5)k$. Then   $H$ is a proper  subgraph of $H''$.
Consider the quotient matrix $Q$ of $A(H')$:
$$Q=\begin{bmatrix}
2 & p+1 & 0 \\
1 & 0 & 1\\
0 & 1 & 0
\end{bmatrix},$$
whose characteristic polynomial is $$\varphi_Q(\lambda):=|\lambda I-Q|=\lambda^3-2\lambda^2-(p+2)\lambda+2.$$

Since $\varphi_Q(1+\sqrt{p+3})=2$ and $\varphi'_Q(\lambda)=3\lambda^2-4\lambda-(p+2)>0$ for all $\lambda\ge 1+\sqrt{p+3}$.
We have $\rho(Q)<1+\sqrt{p+3}.$
Applying Lemma \ref{lemma3}, Lemma\ref{lemma11} and Lemma \ref{eqlemma}, we have
$$\rho(H)<\rho(H'')<\rho(H')\le \rho(Q)<1+\sqrt{p+3}.$$
\end{proof}

\section{Proof of Theorem 1}

Suppose  $G\in \mathcal{G}_{n,n-k}$ has a maximum dissociation set $D$. By Lemma \ref{lemma4}, $G$ is a tree. Denote by $V(G)\setminus D=\{v_1,v_2,\ldots,v_k\}$.
Suppose there are exactly $a_i$ pendant leaves and $b_i$ pendant $P_2$ hanging on $v_i$ for $i=1,2,\dots,k$.
Now, we consider the structure of $H[G,D]$ and $G$.

\begin{claim}\label{claim1}
{\it For $i=1,2,\dots,k$, $v_i$ is adjacent to at least $k+3$ components of $G[D]$ in $G$.
Consequently, we have $d_G(v_i)\ge k+3$ for $i=1,2,\dots,k$ and $\rho(G)>\sqrt{k+3}$.}
\end{claim}

\begin{proof}
Assume that $v_i$ is adjacent to at most $k+2$ components of $G[D]$.
Since  $G[D]$ has at least $(n-k)/2$ components and each component is adjacent to at least one of $v_1,v_2,\dots,v_k$,
there exists $j\in\{1,\ldots,k\}\setminus\{i\}$ such that $v_j$ is adjacent to at least $t=\lceil(n-3k-4)/(2k-2)\rceil$ components of $G[D]$,
which implies that $G$ contains a subgraph $S_t$. By  Lemma \ref{lemma13}, we get $$\rho(G)\ge \sqrt{(n-3k-4)/(2k-2)} .$$
On the other hand, applying  Lemma \ref{lem2}, we have $$\rho(G)< 1+\sqrt{p+3} {\rm \quad with\quad  } p=\lfloor(n-3k+1)/(2k)\rfloor.$$
For $n\ge 16k^3$ and $k\ge 4$,
Let $$f(n)=1+\sqrt{\frac{n+3k+1}{2k}},\quad g(n)=\sqrt{\frac{n-3k-4}{2k-2}}.$$ Notice that
 $g(16k^3)-f(16k^3)>0 $
and $$g'(n)-f'(n)=\frac{1}{\sqrt{(n-3k-4)(2k-2)}}-\frac{1}{\sqrt{(n+3k+1)(2k)}}>0. $$
We have $1+\sqrt{p+3}\le f(n)<\sqrt{ {(n-3k-4)}/{(2k-2)}},$
a contradiction. Hence,  $v_i$ is adjacent to at least $k+3$ components of $G[D]$ in $G$,
which implies that $D$ contains a proper subgraph   $S_{k+3}$ and  $\rho(G)>\rho(S_{k+3})=\sqrt{k+3}$.
\end{proof}

\begin{claim}\label{claim2}
{\it The hypergraph $H[G, D]$ is linear and  acyclic.
Consequently, in the graph $G$, any two components of $G[D]$ have at most one common neighbour in $V\setminus D$.}\\
\end{claim}

\begin{proof}
Suppose there exist two edges $e,f\in E'$  both incident to two distinct vertices $v_1,v_2$.
Then  $G$ contains internal disjoint paths $P_1,P_2$ with ends $v_1$ and $v_2$, which leads to
  a cycle $v_1P_1v_2P_2v_1$ in $G$. This contradicts the fact that $G$ is a tree. Therefore, $H[G, D]$ is a linear hypergraph.

Similarly, suppose $H[G,D]$ contains a cycle, say, $v_1e_1v_2e_2\ldots e_{r-1}v_re_rv_1$.
Then $G$ contains  internal disjoint paths $P_i$ with ends  $v_{i}$ and $v_{i+1}$   for $i=1,2,\dots,r$.
Hence,   $G$ contains a cycle $v_1P_1v_2P_2\ldots P_{r-1}v_rP_{r}v_1$, a contradiction. Therefore, $H[G, D]$ is   acyclic.
\end{proof}

\begin{claim}\label{claim3}
{\it For every $v_i$, there are at least 4 components of $G[D]$ hanging on $v_i$ in $G$.}
\end{claim}

\begin{proof}
By Claim \ref{claim1} and Claim \ref{claim2},   $v_i$ is  adjacent to at least $k+3$ components of $G[D]$ in $G$. Since $G$ is a tree, every pair of vertices in $V\setminus D$ has at most one common adjacent component of $G[D]$. Therefore,  in the graph $G$,
at most $k-1$ components of $G[D]$ are adjacent to $v_i$ and another vertex from $V\setminus D$ simultaneously. It follows that   at least 4 components of $G[D]$ are not adjacent to any vertex of $V\setminus D\cup \{v_i\}$, which means that  they hang on $v_i$ in $G$.
\end{proof}

\begin{claim}\label{claim5}
{\it  If $e$ is an edge in $H[G,D]$, then $e$ is  {Type (III)}.
Moreover, If $e$ is incident to exactly two vertices $v_i,v_j$, then $S[G; v_i, v_j]$ is a path  $P_4$ with ends $v_i$ and $v_j$.}
\end{claim}

\begin{proof}
Let $e$ be an edge in   $H[G,D]$. Suppose $e$ is incident to $t$ vertices $v_1, v_2, \ldots, v_t$. We distinguish two cases.

\textbf{Case 1.} $t\ge 3$. Clearly,  $e$ is not  Type (I).
Suppose $e$ is   {Type (II)}, say,  $v_1, v_2, \ldots, v_t$  are adjacent to $u_1\in D$. By Claim \ref{claim2}, each vertex in $D\setminus\{u_1\}$ has at most one neighbour from  $\{v_1, v_2, \ldots, v_t\}$. Therefore, by Definition \ref{defi1},
 $S[G; v_1, v_2, \ldots, v_t]$ is a star $H_1$ (Figure \ref{fig1}).
By Claim \ref{claim3}, $G$  has at least 4 leaves in $N[N(v_1)]\cap D$.
Let $G'$ be the graph obtained from $G$ by deleting a leaf $u\in N[N(v_1)]\cap D$ and subdividing the edge $v_1u_1$.
By Lemma \ref{lemma3} and Lemma \ref{lemma11}, we have $\rho(G')<\rho(G)$.
Moreover, we have $d_{G'}(v_i)=d_{G}(v_i)$ for $i=2,\dots,k$ and $d_{G}(v_1)-1\le d_{G'}(v_1)\le d_{G}(v_1)$. It follows from  Claim \ref{claim1} that  $d_{G'}(v_i)\ge k+1$ for $i=1,2,\dots,k$. Notice that $V(G')-\{v_1, \dots,v_k\}$ is a dissociation set of $G'$.
Applying Lemma \ref{lem1}   we have ${\rm diss}(G')=n-k$ with $\rho(G')<\rho(G)$, which contradicts $G \in \mathcal{G}_{n,n-k}$.  Hence, $e$ is Type (III).

\textbf{Case 2.}  $t=2$. If $e$ is    {Type (I)}, then $S[G; v_1, v_2]$ is an   $H_2$ (Figure \ref{fig1}).
Let $G'$ be the graph obtained from $G$ by deleting two leaves of $G$ from $N[N(v_1)]\cap D$ and doing a 2-subdivision to the edge $v_1v_2$.

If $e$ is   {Type (II)}, then $S[G; v_1, v_2]$ is isomorphic to $H_3$ (Figure \ref{fig1}).
Let $G'$ be the graph obtained from $G$ by deleting a leaf of $G$ from $N[N(v_1)]\cap D$ and subdividing the edge $u_1v_1$.

If $e$ is    {Type (III)}, then $S[G; v_1, v_2]$ is isomorphic to $H_4$ or $H_5$ (Figure \ref{fig1}).
In the former case, let $G'$ be the graph obtained from $G$ by deleting $u^2_1$ and subdividing the edge $v_1u^1_1$.

Similarly with in Case 1, in the above subcases, we have $d_{G'}(v_i)=d_{G}(v_i)$ for  $i=2,\dots,k$ and $d_{G}(v_i)-2\le d_{G'}(v_i)\le d_{G}(v_i)$, which leads to
$\rho(G')<\rho(G)$ and ${\rm diss}(G')=n-k$, a contradiction. Therefore,  $e$ is Type (III) and $S[G; v_1, v_2]$ is a path  $P_4$ with ends $v_1$ and $v_2$.
\end{proof}
\begin{figure}[H]
	\centering
    \begin{subfigure}{0.15\textwidth}
	\centering
	\begin{tikzpicture}
	\draw(0,0)[fill]circle[radius=1.0mm];\draw(1,0.5)[fill]circle[radius=1.0mm];\draw(1,1.5)[fill]circle[radius=1.0mm];\draw(1,-1)[fill]circle[radius=1.0mm];
    \draw(0,0)--(1,0.5);\draw(0,0)--(1,1.5);\draw(0,0)--(1,-1);
    \node [left] at (0,0) {$u_1$};\node [right] at (1,1.5) {$v_1$};\node [right] at (1,0.5) {$v_2$};\node [right] at (1,-1) {$v_t$};\node at (1,-0.25) {$\vdots$};
    \node at(0.5,-1.5) {$H_1$};
	\end{tikzpicture}
	\end{subfigure}
	\begin{subfigure}{.15\textwidth}
	\centering
    \vspace{21pt}
	\begin{tikzpicture}
    \draw(-1,0)[fill]circle[radius=1.0mm];\draw(0,0)[fill]circle[radius=1.0mm];
    \draw(-1,0)--(0,0); \node [above] at (-1,0) {$v_1$}; \node [above] at (0,0) {$v_2$};
    \node at (-0.5,-1.5) {$H_2$};
	\end{tikzpicture}
	\end{subfigure}
    \begin{subfigure}{.2\textwidth}
	\centering
    \vspace{21pt}
	\begin{tikzpicture}
	\draw(-1,0)[fill]circle[radius=1.0mm];\draw(0,0)[fill]circle[radius=1.0mm];\draw(1,0)[fill]circle[radius=1.0mm];
    \draw(-1,0)--(0,0);\draw(1,0)--(0,0);\node [above] at (-1,0) {$v_1$}; \node [above] at (1,0) {$v_2$}; \node [above] at (0,0) {$u_1$};
    \node at (0,-1.5) {$H_3$};
	\end{tikzpicture}
    \end{subfigure}
    \begin{subfigure}{.22\textwidth}
	\centering
	\begin{tikzpicture}
    \draw(-1,0)[fill]circle[radius=1.0mm];\draw(0,0)[fill]circle[radius=1.0mm];\draw(1,0)[fill]circle[radius=1.0mm];\draw(0,1)[fill]circle[radius=1.0mm];
    \draw(-1,0)--(0,0); \draw(1,0)--(0,0);\draw(0,0)--(0,1);\node [above] at (-1,0) {$v_1$}; \node [above] at (1,0) {$v_2$};
    \node [above] at (0.3,0) {$u^1_1$}; \node [right] at (0,1) {$u^2_1$};
    \node at (0,-1.5) {$H_4$};
	\end{tikzpicture}
	\end{subfigure}
    \begin{subfigure}{.22\textwidth}
	\centering
    \vspace{21pt}
	\begin{tikzpicture}
    \draw(-1,0)[fill]circle[radius=1.0mm];\draw(0,0)[fill]circle[radius=1.0mm];\draw(1,0)[fill]circle[radius=1.0mm];\draw(2,0)[fill]circle[radius=1.0mm];
    \draw(-1,0)--(0,0); \draw(0,0)--(1,0); \draw(1,0)--(2,0);
    \node [above] at (-1,0) {$v_1$}; \node [above] at (2,0) {$v_2$}; \node [above] at (0,0) {$u^1_1$}; \node [above] at (1,0) {$u^2_1$};
    \node at (0.5,-1.5) {$H_5$};
	\end{tikzpicture}
	\end{subfigure}
 \caption{ }
    \label{fig1}
\end{figure}

\begin{claim}\label{Claim5}
 Let $\rho=\rho(G)$ and
 $X = (X_1, X_2)^T$ be the Perron vector of $G$,
where the components of $X_1$ corresponds to $v_1,v_2,\dots, v_k$ and  the components of  $X_2$ corresponds to the vertices in $D$.
Then we have $$\rho(\rho^2-1)X_1=M_{G,D}(\rho)X_1.$$
Moreover, $M_{G,D}(t)$ is an irreducible nonnegative matrix for $t>1$ and
 $\lambda_1(M_{G,D}(t))=O(t)$.
\end{claim}

\begin{proof}

For each vertex $v_i \in V(G)\setminus D$, we denote by $x_i$, $y_i$, $z_i$ and $w_i$   the   components of $X$, where $x_i$ corresponds to $v_i$, $y_i$ corresponds to the closer end of $P_2$ attached on $v_i$, $z_i$ corresponds to the leaves   at distance 2 from $v_i$, and   $w_i$   corresponds to the leaves attached on $v_i$ (Figure 2).

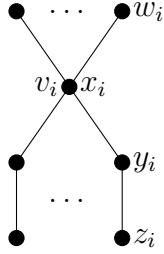
\begin{figure}[H]
	\centering
    \begin{subfigure}{\textwidth}
	\centering
	\begin{tikzpicture}
	\draw(0,0)[fill]circle[radius=1.0mm];\draw(0.7,1)[fill]circle[radius=1.0mm];\draw(-0.7,1)[fill]circle[radius=1.0mm];
    \draw(0.7,-1)[fill]circle[radius=1.0mm];\draw(-0.7,-1)[fill]circle[radius=1.0mm];\draw(0.7,-2)[fill]circle[radius=1.0mm];\draw(-0.7,-2)[fill]circle[radius=1.0mm];
    \draw(0,0)--(0.7,1);\draw(0,0)--(0.7,-1);\draw(0,0)--(-0.7,1);\draw(0,0)--(-0.7,-1);\draw(-0.7,-2)--(-0.7,-1);\draw(0.7,-2)--(0.7,-1);
    \node [left] at (0,0) {$v_i$};\node [right] at (0,0) {$x_i$};\node [right] at (0.7,1) {$w_i$};\node [right] at (0.7,-1) {$y_i$};\node [right] at (0.7,-2) {$z_i$};
    \node at (0,1) {$\dots$}; \node at (0,-1.5) {$\dots$};
	\end{tikzpicture}
	\end{subfigure}
\caption{The   components of $X$ corresponding to pendant leaves and pendant $P_2$ on $v_i$}
\end{figure}

Let $\mathcal{C}$ be the set of all components in $G[D]$ that are adjacent to at least two of $v_1, v_2, \dots, v_k$.
Since all edges of $H[G,D]$ are  {Type (III)},
every component in $\mathcal{C}$ is an edge.
We define an index set
\[
\mathcal{I} = \{(P,Q) \mid  \ P = N_G(u) -V(\mathcal{C}),\ Q = N_G(u') - V(\mathcal{C}), \text{ where } uu'\in \mathcal{C}\}.
\]
For convenience, we denote by $x_{P,Q}$ and $x_{Q,P}$ the components of $X$ corresponding to $u$ and $u'$, respectively; see Figure 3.

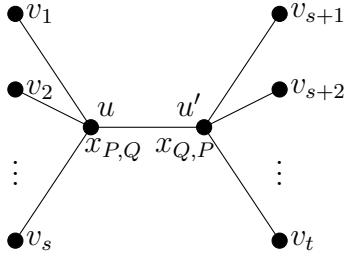
\begin{figure}[H]
	\centering
    \begin{subfigure}{0.45\textwidth}
	\centering
	\begin{tikzpicture}
	\draw(0.5,0)[fill]circle[radius=1.0mm];\draw(1.5,0.5)[fill]circle[radius=1.0mm];\draw(1.5,1.5)[fill]circle[radius=1.0mm];\draw(1.5,-1.5)[fill]circle[radius=1.0mm];
    \draw(-1,0)[fill]circle[radius=1.0mm];\draw(-2,0.5)[fill]circle[radius=1.0mm];\draw(-2,1.5)[fill]circle[radius=1.0mm];\draw(-2,-1.5)[fill]circle[radius=1.0mm];
    \draw(0.5,0)--(1.5,0.5);\draw(0.5,0)--(1.5,1.5);\draw(0.5,0)--(1.5,-1.5);\draw(-1,0)--(0.5,0);
    \draw(-1,0)--(-2,0.5);\draw(-1,0)--(-2,1.5);\draw(-1,0)--(-2,-1.5);
    \node [below] at (0.25,0) {$x_{Q,P}$};\node [right] at (1.5,1.5) {$v_{s+1}$};\node [right] at (1.5,0.5) {$v_{s+2}$};\node [right] at (1.5,-1.5) {$v_t$};\node at (1.5,-0.5) {$\vdots$};
    \node [below] at (-0.7,0) {$x_{P,Q}$};\node [right] at (-2,1.5) {$v_1$};\node [right] at (-2,0.5) {$v_2$};\node [right] at (-2,-1.5) {$v_{s}$};\node at (-2,-0.5) {$\vdots$};
    \node [above] at (0.3,0) {$u'$};\node [above] at (-0.8,0) {$u$};
	\end{tikzpicture}
	\end{subfigure}
\caption{Neighbors of $u$ and $u'$}
\end{figure}

Now   considering the characteristic equation of  $D$, we have
$$
\left\{
\begin{aligned}
&\rho y_i=x_i+z_i,\\
&\rho z_i=y_i,\\
&\rho w_i=x_i,\\
&\rho x_{P,Q}=x_{Q,P}+\sum_{\{j:v_j\in P\}}x_j,
\end{aligned}
\right.
$$
which lead  to
$$
\left\{
\begin{aligned}
&y_i=\frac{\rho}{\rho^2-1}x_i,\\
&z_i=\frac{1}{\rho^2-1}x_i,\\
&w_i=\frac{x_i}{\rho},\\
&x_{P,Q}=\frac{1}{\rho^2-1}(\rho\sum_{\{j:v_j\in P\}}x_j+\sum_{\{j:v_j\in Q\}}x_j).
\end{aligned}
\right.
$$

For $i=1,2,\ldots,k$, considering the characteristic equation corresponding to $v_i$, we have

\[\begin{aligned}
\rho(\rho^2-1)x_i
&=(\rho^2-1)(a_iw_i+b_iy_i+\sum_{\{(P,Q)\in \mathcal{I}:v_i\in P\}}x_{P,Q})\\
&=a_i\frac{(\rho^2-1)x_i}{\rho}+b_i\rho x_i+\sum_{\{(P,Q)\in \mathcal{I}:v_i\in P\}}(\rho \sum_{\{j:v_j\in P\}}x_j+\sum_{\{j:v_j\in Q\}}x_j)\\
&=a_i\frac{(\rho^2-1)x_i}{\rho}+b_i\rho x_i+\sum_{\{(P,Q)\in \mathcal{I}:v_i\in P\}}(\rho x_i+\rho\sum_{\{j:v_j\in P,j\neq i\}}x_j+\sum_{\{j:v_j\in Q\}}x_j)\\
&=a_i\frac{(\rho^2-1)x_i}{\rho}+b_i\rho x_i+\rho(d_G(v_i)-a_i-b_i)x_i+\rho\sum_{\{j:d(v_i,v_j)=2\}}x_j+\sum_{\{j:d(v_i,v_j)=3\}}x_j\\
&=(\rho d_G(v_i)-\frac{a_i}{\rho})x_i+\rho\sum_{\{j:d(v_i,v_j)=2\}}x_j+\sum_{\{j:d(v_i,v_j)=3\}}x_j,
\end{aligned}\]
which leads to $\rho(\rho^2-1)X_1=M_{G,D}(\rho)X_1$.
When $t>1$, an off-diagonal entry $m_{ij}$ of $M_{G,D}(t)$ is positive if and only if $v_i$ is adjacent to $v_j$ in  $H[G,D]$.
Since $H[G,D]$ is connected, $M_{G,D}(t)$ is an irreducible nonnegative matrix.
Moreover, the $i$-th row sum of $M_{G,D}(t)$ is $O(t)$ for $i=1,2,\ldots,k$. Hence,
we have $\lambda_1(M_{G,D}(t))=O(t)$.

\end{proof}

\begin{claim}\label{Claim6}
{\it $a_i\in \{0,1\}$ for $i=1,2,\dots,k$}.
\end{claim}

\begin{proof}
 To the contrary, suppose $a_i\ge 2$ and   $w_1,w_2$ are two leaves attached on   $v_i$. Let $G'$ be the graph obtained from $G$ by deleting the edge $w_1v_i$ and adding  an edge $w_1w_2$.
By Lemma \ref{lemma14}, $\rho(G')<\rho(G)$.   Notice that $d_G'(v_i)=d_G(v_i)-1$  and Claim 1 guarantees  $d_G(v_j)\ge k+3$ for $j=1,2,\ldots,k$.
Applying Lemma \ref{lem1}, we have ${\rm diss}(G)\le n-k$.
Since  $D=V\setminus \{v_1,v_2,\dots,v_k\}$ is also a dissociation set of $G'$,  we have ${\rm diss}(G)=n-k$,
which contradicts  $G \in \mathcal{G}_{n,n-k}$. Therefore, we get $a_i\in\{0,1,\}$.
\end{proof}

\begin{claim}\label{lem3}
Let $G_1$ and $G_2$ be two $n$-vertex graphs sharing the same vertex set $V$ and a common maximum dissociation set $D_1$ such that\\
\indent (1) $H[G_1,D_1],H[G_2,D_1]$ are linear hypergraphs;\\
\indent (2) all edges of $H[G_1,D_1],H[G_2,D_1]$ are  {Type (III)}.\\
 If
$\lambda_1(M_{G_1,D_1}(t))>\lambda_1(M_{G_2,D_1}(t))$ for $t\ge\min\{\rho(G_1),\rho(G_2)\}$ and
$\lim\limits_{t\to\infty}(t(t^2-1)-\lambda_1(M_{G_i,D_1}(t)))=+\infty$ for $i=1,2$, then $\rho(G_1)>\rho(G_2)$.
\end{claim}

\begin{proof}
 Denote by  $\rho_1=\rho(G_1)$ and $\rho_2=\rho(G_2)$.
Let    $X=(X_1,X_2)^T$ and $Y=(Y_1,Y_2)^T$ be the Perron vectors of $G_1$ and $G_2$, respectively,  where $X_1,Y_1$ corresponds to the vertices in $V\setminus D$. Since ${\rm diss}(G_1)={\rm diss}(G_2)=n-k$, by  Lemma \ref{claim0}, we have $\min\{\rho_1,\rho_2\}>1$.
Applying the same arguments as in Claim \ref{Claim5}, we get $$\rho_1(\rho_1^2-1)X_1=M_{G_1,D_1}(\rho_1)X_1\quad \text{and}\quad \rho_2(\rho_2^2-1)Y_1=M_{G_2,D_1}(\rho_2)Y_1.$$
 Moreover, $M_{G_1,D_1}(t)$ and $M_{G_2,D_1}(t)$ are irreducible nonnegative matrices for $t\ge\min\{\rho_1,\rho_2\}$.
Since $X_1,Y_1>0$,
  by Lemma \ref{PF}, $X_1$ and $Y_1$ are eigenvectors of $M_{G_1,D_1}(\rho_1)$ and $M_{G_2,D_1}(\rho_2)$ corresponding to their maximum   eigenvalues.
  Thus, we have $$\lambda_1(M_{G_1,D_1}(\rho_1))=\rho_1(\rho_1^2-1) \quad \text{ and }\quad \lambda_1(M_{G_2,D_1}(\rho_2))= \rho_2(\rho_2^2-1).$$
  Let $f(t)=t(t^2-1)$, $h_1(t)=\lambda_1(M_{G_1,D_1}(t))$ and $h_2(t)=\lambda_1(M_{G_2,D_1}(t))$ for $t\ge\min\{\rho_1,\rho_2\}$. By Lemma \ref{lemma15}, we have $\rho_1>\rho_2$.
\end{proof}

\begin{claim}\label{claim8}
{\it $H[G, D]$ is a 2-uniform graph.}
\end{claim}

\begin{proof}
To the contrary, suppose $e\in E(H)$ is incident  to $q\ge 3$ vertices $v_1,v_2,\ldots,v_q$,  and $S[G;v_1,v_2,\ldots,v_q]$ has the following diagram.
\begin{figure}[H]
	\centering
    \begin{subfigure}{0.45\textwidth}
	\centering
	\begin{tikzpicture}
	\draw(0.5,0)[fill]circle[radius=1.0mm];\draw(1.5,0.5)[fill]circle[radius=1.0mm];\draw(1.5,1.5)[fill]circle[radius=1.0mm];\draw(1.5,-1.5)[fill]circle[radius=1.0mm];
    \draw(-1,0)[fill]circle[radius=1.0mm];\draw(-2,0.5)[fill]circle[radius=1.0mm];\draw(-2,1.5)[fill]circle[radius=1.0mm];\draw(-2,-1.5)[fill]circle[radius=1.0mm];
    \draw(0.5,0)--(1.5,0.5);\draw(0.5,0)--(1.5,1.5);\draw(0.5,0)--(1.5,-1.5);\draw(-1,0)--(0.5,0);
    \draw(-1,0)--(-2,0.5);\draw(-1,0)--(-2,1.5);\draw(-1,0)--(-2,-1.5);
    \node [right] at (1.5,1.5) {$v_{p+1}$};\node [right] at (1.5,0.5) {$v_{p+2}$};\node [right] at (1.5,-1.5) {$v_q$};\node at (1.5,-0.5) {$\vdots$};
    \node [right] at (-2,1.5) {$v_1$};\node [right] at (-2,0.5) {$v_2$};\node [right] at (-2,-1.5) {$v_{p}$};\node at (-2,-0.5) {$\vdots$};
    \node at(-0.25,-2) {$S[G;v_1,v_2,\cdots,v_q]$};\node [above] at (0.3,0) {$u'$};\node [above] at (-0.8,0) {$u$};
	\end{tikzpicture}
	\end{subfigure}
    \begin{subfigure}{0.45\textwidth}
	\centering
	\begin{tikzpicture}
	\draw(-3,0)[fill]circle[radius=1.0mm];\draw(-2.5,0)[fill]circle[radius=1.0mm];\draw(-2,0)[fill]circle[radius=1.0mm];\draw(-1.5,0)[fill]circle[radius=1.0mm];
    \draw(-1,0)[fill]circle[radius=1.0mm];\draw(-0.5,0)[fill]circle[radius=1.0mm];\draw(0,0)[fill]circle[radius=1.0mm];\draw(1.5,0)[fill]circle[radius=1.0mm];
    \draw(-3,0)--(-2.5,0);\draw(-2.5,0)--(-2,0);\draw(-2,0)--(-1.5,0);\draw(-1.5,0)--(-1,0);\draw(-1,0)--(-0.5,0);\draw(-0.5,0)--(0,0);
    \draw[dotted](0,0)--(1,0);\draw(1,0)--(1.5,0);\draw(1,0)[fill]circle[radius=1.0mm];
    \node[below] at (-3,0) {$v_1$};\node[below] at (-1.5,0) {$v_2$};\node[below] at (0,0) {$v_3$};\node[below] at (1.5,0) {$v_q$};
    \node at(-0.5,-1) {$S[G';v_1,v_2,\cdots,v_q]$};
	\end{tikzpicture}	
    \end{subfigure}
    \caption{}
\end{figure}

By Claim \ref{claim3} and Claim \ref{Claim6},   each $v_i$ is adjacent to at least one pendant $P_2$. Choose a pendant $P_2$ attached on  each $v_i$ for $i=2,\ldots,q$ and denote the vertices of these $P_2$ and $S[G;v_1,v_2,\cdots,v_q]$ by $V_1$.
We  construct a graph $G'$ from $G$ by  replacing the induced subgraph $G[V_1]$ with a path $P_{3q-2}$, in which $d(v_i,v_{i+1})=3$ for $i=1,\ldots,q-1$. Then $S[G';v_1,v_2,\cdots,v_q]$ is the added path $P_{3q-2}$; see Figure 4.

Notice that $d_G'(v_i)=d_G(v_i)$ for $i=1,2,\dots,k$ and $V(G)\setminus \{v_1,v_2,\dots, v_k\}$ is also a dissociation set of $G'$ with size $n-k$.
By Lemma \ref{lem1} and Claim \ref{claim1}, we have ${\rm diss}(G')=n-k$. Moreover,
 $H[G',D]$ is a linear hypergraph
 and all edges of $H[G',D]$ are Type (III).

Now we consider $M_{G,D}(t)$ and $M_{G',D}(t)$  for $t>1$.
Partition $M_{G,D}$ and $M_{G',D}$  as
$$
M_{G,D}(t) = \begin{bmatrix}
M_{11} & M_{12} \\
M_{21} & M_{22}
\end{bmatrix}, \quad M_{G',D}(t) = \begin{bmatrix}
M'_{11} & M'_{12} \\
M'_{21} & M'_{22}
\end{bmatrix},
$$
where $M_{11}$ and $M'_{11}$ are $q\times q$ matrices  corresponding  to $v_1,v_2,\ldots,v_q$.
Notice that the only difference between $G$ and $G'$ is $G[V_1]$ and $G'[V_1]$. We have   $M_{12}=M'_{12}$, $M_{21}=M'_{21}$, $M_{22}=M'_{22}$ and

$M_{11}(t)=$\scalebox{0.7}{$\begin{bmatrix}
d_G(v_1)t-\frac{a_1}{t} & t & \cdots & t & 1 & 1 & \cdots & 1 \\
t & d_G(v_2)t-\frac{a_2}{t} & \cdots & t & 1 & 1 & \cdots & 1 \\
\vdots & \vdots & \ddots & \vdots & \vdots & \vdots & \ddots & \vdots \\
t & t & \cdots & d_G(v_p)tp-\frac{a_p}{t} & 1 & 1 & \cdots & 1 \\
1 & 1 & \cdots & 1 & d_G(v_{p+1})t-\frac{a_{p+1}}{t} & t & \cdots & t \\
1 & 1 & \cdots & 1 & t & d_G(v_{p+2})t-\frac{a_{p+2}}{t} & \cdots & t \\
\vdots & \vdots & \ddots & \vdots & \vdots & \vdots & \ddots & \vdots \\
1 & 1 & \cdots & 1 & t & t & \cdots & d_G(v_q)-\frac{a_q}{t}
\end{bmatrix}$,}

$M'_{11}(t)=$\scalebox{0.64}{$\begin{bmatrix}
d_G(v_1)t - \dfrac{a_1}{t} & 1 \\
1 & d_G(v_2)t - \dfrac{a_2}{t} & 1 \\
& 1 & \ddots & \ddots \\
& & \ddots & \ddots & 1 \\
& & & 1 & d_G(v_p)t - \dfrac{a_p}{t} & 1 \\
& & & & 1 & d_G(v_{p+1})t - \dfrac{a_{p+1}}{t} & 1 \\
& & & & & 1 & \ddots & \ddots \\
& & & & & & \ddots & \ddots & 1 \\
& & & & & & & 1 & d_G(v_{q-1})t - \dfrac{a_{q-1}}{t} & 1 \\
& & & & & & & & 1 & d_G(v_q)t - \dfrac{a_q}{t}
\end{bmatrix}$.}

Since $M_{G,D}(t),M_{G',D}(t)$ are irreducible non-negative matrix for $t>1$,
$M_{G,D}(t)\ge M_{G',D}(t)$ and $M_{G,D}(t)\neq M_{G',D}(t)$, by Lemma \ref{MS}, we have $\lambda_1(M_{G,D}(t))>\lambda_1(M_{G',D}(t))$. Notice that $\lambda_1(M_{G,D}(t))=O(t)$ and  $\lambda_1(M_{G',D}(t))=O(t)$. Applying Claim  \ref{lem3} on $G$ and $G'$ we have
$\rho(G)>\rho(G')$, a contradiction. Hence, $H[G,D]$ is a 2-uniform graph.
\end{proof}

\begin{claim}\label{claim9}
{\it $d_G(v_1), d_G(v_2),\ldots,d_G(v_k)$ differ at most 1.}
\end{claim}

\begin{proof}
By Claim \ref{claim2}, Claim \ref{claim5} and Claim \ref{claim8}, we have $$M_{G,D}(t)=\text{diag}\left(d_G(v_1)t-\frac{a_1}{t},d_G(v_2)t-\frac{a_2}{t},\ldots,d_G(v_k)t-\frac{a_k}{t}\right)+A(T_k),$$
where $A(T_k)$ is the adjacency matrix of a tree $T_k$ with $V(T_k)=\{v_1,v_2,\dots, v_k\}$ and $v_i$ is adjacent to $v_j$ in $T_k$ if and only if $d_G(v_i,v_j)=3$.

Since $M_{G,D}(t)$ is a  symmetric nonnegative matrix,
by Lemma \ref{W_ieq}, Lemma \ref{lemma13} and Claim \ref{Claim6},  we have
\begin{equation}\label{eq1}
\begin{aligned}
	&\max\{d_G(v_i)|i\in[k]\}t-\frac{1}{t} \leq \rho(M_{G,D}(t)) \leq \max\{d_G(v_i)|i\in[k]\}t+\sqrt{k}.
\end{aligned}
\end{equation}

Suppose $d_G(v_i),i=1,2,\dots,k$ differ at least 2. We   construct a new graph $G'$  from $G$ by moving pendant edges on $v_1,\ldots,v_k$ to make $d_G'(v_1), \dots,d_G'(v_k)$ differ at most 1. Similarly with \eqref{eq1}, we have
\begin{equation}\label{eq2}
\begin{aligned}
	&\max\{d_{G'}(v_i)|i\in[k]\}t-\frac{1}{t} \leq \rho(M_{G',D}(t)) \leq \max\{d_{G'}(v_i)|i\in[k]\}t+\sqrt{k}.
\end{aligned}
\end{equation}
Since $\max\{d_{G}(v_i)|i\in[k]\}-\max\{d_{G'}(v_i)|i\in[k]\}\ge 1$, we have
$$\rho(M_{G',D}(t)) \leq \max\{d_{G'}(v_i)|i\in[k]\}t+\sqrt{k}<\max\{d_{G}(v_i)|i\in[k]\}t-\frac{1}{t} \leq \rho(M_{G,D}(t))$$
for $t>\sqrt{k+3}$.

By Claim \ref{claim1}, we have $\rho(G)\ge \sqrt{k+3}$ and $\rho(G')\ge \sqrt{k+3}$.
Notice that  $\rho(M_{G,D}(t))=O(t)$ and $\rho(M_{G',D}(t))=O(t)$.
By Claim \ref{lem3}, we have $\rho(G')<\rho(G)$, a contradiction. Thus we get the claim.
\end{proof}

\begin{claim}
{\it $G\in D_n(\mathcal{T}_k)$ with $ \Delta(G)=\lceil {(n+k-2)}/{(2k)}\rceil.$}
\end{claim}

\begin{proof}
Since $n\ge 16k^3$, combining Claim \ref{Claim6}, Claim \ref{claim8} and Claim \ref{claim9} we can conclude that $G\in D_n(\mathcal{T}_k)$. Moreover, applying the same arguments as in Claim \ref{claim9}, all graphs in $\mathcal{G}_{n,n-k}$ have the same maximum degree $\Delta$.

 Now we consider the value of $\Delta(G)$. Since each edge incident to $v_i$ is connected to a $P_2$ or a leaf with vertices from $V(G)\setminus D$. We have
 $$|V(G)|\le k+2\sum_{i\in [k]}d_G(v_i)-2(k-1)\le k+2k\Delta(G)-2(k-1)=2k\Delta(G)-k+2$$
 which leads to $\Delta(G)\ge \lceil (n+k-2)/(2k)\rceil$.

 On the other hand, let $G'$ be a graph in $ D_n(P_k)\subset D_n(\mathcal{T}_k)$. Then $diss(G')=n-k$ and $\Delta(G')= \lceil (n+k-2)/(2k)\rceil$.  If $\Delta (G)>\lceil (n+k-2)/(2k)\rceil$, then applying the same arguments as in Claim \ref{claim9} we can deduce $\rho(G)>\rho(G')$, a contradiction. Therefore, we have $\Delta(G)= \lceil (n+k-2)/(2k)\rceil$.
\end{proof}

\begin{claim}
{\it Suppose  $n$ is sufficiently large. For any $G_1,G_2\in D_n(\mathcal{T}_k)$ with maximum degree $\lceil (n+k-2)/(2k)\rceil$, we have
\begin{equation}\label{eq6}
|\rho(G_1)-\rho(G_2)|\le  \frac{2(k^{3/2}+k)}{3n}.
\end{equation}}
\end{claim}

\begin{proof}
Denote by $\rho_1=\rho(G_1),\rho_2=\rho(G_2)$, $\Delta=\lceil (n+k-2)/(2k)\rceil$.  Suppose $D_1$ and $D_2$ are maximum dissociation sets of $G_1$ and $G_2$, respectively.  Without loss of generality, we may assume  $\rho_1\le \rho_2$. Applying the same arguments as in Claim \ref{Claim5} and  Claim \ref{claim9}, we have
\begin{equation}\label{eq2}
\begin{aligned}
	\Delta{\rho_1}-\frac{1}{\rho_1}\le \rho_1(\rho_1^2-1)=\lambda_1(M_{G_1,D_1}(\rho_1))\le \Delta\rho_1+\sqrt{k},
\end{aligned}
\end{equation}
\begin{equation}\label{eq3}
\begin{aligned}
	\Delta\rho_2-\frac{1}{\rho_2}\le \rho_2(\rho_2^2-1)=\lambda_1(M_{G_2,D_2}(\rho_2))\le \Delta\rho_2+\sqrt{k},
\end{aligned}
\end{equation}
which implies
 $$\rho_2(\rho_2^2-1)-\rho_1(\rho_1^2-1)\leq \Delta(\rho_2-\rho_1)+\sqrt{k}+\frac{1}{\rho_1}\le\sqrt{k}+1.$$
 By the Lagrange's Mean Value Theorem, we have
$$\frac{\rho_2(\rho_2^2-1)-\rho_1(\rho_1^2-1)}{\rho_2-\rho_1}=3\xi^2-1,$$
for some $\xi\in [\rho_1,\rho_2]$, which implies
\begin{equation}\label{eq5}
\rho_2-\rho_1=\frac{\rho_2(\rho_2^2-1)-\rho_1(\rho_1^2-1)}{3\xi^2-1}\le \frac{\sqrt{k}+1}{3\rho_1^2-1}.
\end{equation}
By (\ref{eq2}), we have $$\rho_1(\rho_1^2-1)\ge\left\lceil\frac{n+k-2}{2k}\right\rceil\rho_1-\frac{1}{\rho_1}\ge \frac{n-k-2}{2k}\rho_1,$$
which leads to $\rho_1^2\ge \frac{n+k-2}{2k}$. Substitute this to \eqref{eq5}, we get
$$\rho_2-\rho_1\le \frac{2(k^{3/2}+k)}{3(n+k-2)}\le \frac{2(k^{3/2}+k)}{3n}.$$
Therefore, we have  \eqref{eq6}.

\end{proof}
Now combining Claim 10 and Claim 11 we complete the proof of Theorem \ref{theorem2}.

\end{document}